\documentclass[oneside,american,english,11pt]{amsart}
\usepackage[T1]{fontenc}
\usepackage[utf8x]{inputenc}
\usepackage{babel}
\usepackage{amstext}
\usepackage{amsthm}
\usepackage{amssymb}
\usepackage[unicode=true,pdfusetitle,
 bookmarks=true,bookmarksnumbered=false,bookmarksopen=false,
 breaklinks=false,pdfborder={0 0 1},backref=false,colorlinks=false]
 {hyperref}

\makeatletter
\numberwithin{equation}{section}
\numberwithin{figure}{section}
\theoremstyle{plain}
\newtheorem{thm}{\protect\theoremname}[section]
\theoremstyle{definition}
\newtheorem{defn}[thm]{\protect\definitionname}
\theoremstyle{plain}
\newtheorem{lem}[thm]{\protect\lemmaname}
\theoremstyle{plain}
\newtheorem{prop}[thm]{\protect\propositionname}
\theoremstyle{definition}
\newtheorem{example}[thm]{\protect\examplename}
\theoremstyle{definition}
\newtheorem{problem}[thm]{\protect\problemname}

\usepackage{ae,aecompl} 
\usepackage{fullpage} 

\theoremstyle{plain}
\newtheorem{manualtheoreminner}{Theorem}
\newenvironment{manualtheorem}[1]{%
  \renewcommand\themanualtheoreminner{#1}%
  \manualtheoreminner
}{\endmanualtheoreminner}

\makeatother

\addto\captionsamerican{\renewcommand{\definitionname}{Definition}}
\addto\captionsamerican{\renewcommand{\examplename}{Example}}
\addto\captionsamerican{\renewcommand{\lemmaname}{Lemma}}
\addto\captionsamerican{\renewcommand{\problemname}{Problem}}
\addto\captionsamerican{\renewcommand{\propositionname}{Proposition}}
\addto\captionsamerican{\renewcommand{\theoremname}{Theorem}}
\addto\captionsenglish{\renewcommand{\definitionname}{Definition}}
\addto\captionsenglish{\renewcommand{\examplename}{Example}}
\addto\captionsenglish{\renewcommand{\lemmaname}{Lemma}}
\addto\captionsenglish{\renewcommand{\problemname}{Problem}}
\addto\captionsenglish{\renewcommand{\propositionname}{Proposition}}
\addto\captionsenglish{\renewcommand{\theoremname}{Theorem}}
\providecommand{\definitionname}{Definition}
\providecommand{\examplename}{Example}
\providecommand{\lemmaname}{Lemma}
\providecommand{\problemname}{Problem}
\providecommand{\propositionname}{Proposition}
\providecommand{\theoremname}{Theorem}

\begin{document}
\global\long\def\epsilon{\varepsilon}%

\global\long\def\phi{\varphi}%

\global\long\def\NN{\mathbb{N}}%

\global\long\def\RR{\mathbb{R}}%

\global\long\def\SS{\mathbb{S}}%

\global\long\def\KK{\mathcal{K}}%

\global\long\def\oo{\mathbf{1}}%

\global\long\def\dd{\mathrm{d}}%

\global\long\def\div{\operatorname{div}}%

\global\long\def\cvx{\operatorname{Cvx}}%

\global\long\def\lc{\operatorname{LC}}%

\global\long\def\dom{\operatorname{dom}}%

\global\long\def\HH{\mathcal{H}}%

\global\long\def\bd#1{\overline{#1}}%

\title{A Riesz representation theorem for log-concave functions}
\author{Liran Rotem}
\address{Department of Mathematics, Technion - Israel Institute of Technology,
Israel}
\email{lrotem@technion.edu}
\thanks{The author is partially supported by ISF grant 1468/19 and BSF grant
2016050.}
\begin{abstract}
The classic Riesz representation theorem characterizes all linear
and increasing functionals on the space $C_{c}(X)$ of continuous
compactly supported functions. A geometric version of this result,
which characterizes all linear increasing functionals on the set of
convex bodies in $\RR^{n}$, was essentially known to Alexandrov.
This was used by Alexandrov to prove the existence of mixed area measures
in convex geometry. 

In this paper we characterize linear and increasing functionals on
the class of log-concave functions on $\RR^{n}$. Here ``linear''
means linear with respect to the natural addition on log-concave functions
which is the sup-convolution. Equivalently, we characterize pointwise-linear
and increasing functionals on the class of convex functions. For some
choices of the exact class of functions we prove that there are no
non-trivial such functionals. For another choice we obtain the expected
analogue of the result for convex bodies. And most interestingly,
for yet another choice we find a new unexpected family of such functionals. 

Finally, we explain the connection between our results and recent
work done in convex geometry regarding the surface area measure of
a log-concave functions. An application of our results in this direction
is also given. 
\end{abstract}

\maketitle

\section{\label{sec:introduction}Introduction}

The Riesz (or Riesz–Markov–Kakutani) representation theorem is the
following classic result of functional analysis:
\begin{thm}
\label{thm:riesz-classic}Let $X$ be a locally compact Hausdorff
space. Let $C_{c}(X)$ denote the class of all continuous and compactly
supported functions $f:X\to\RR$. Let $F:C_{c}(X)\to\RR$ be a functional
such that:
\begin{enumerate}
\item $F$ is linear: $F(\alpha f+\beta g)=\alpha F(f)+\beta F(g)$ for
all $f,g\in C_{c}(X)$ and $\alpha,\beta\in\RR$.
\item $F$ is increasing: If $f,g\in C_{c}(X)$ and $f\ge g$ then $F(f)\ge F(g)$.
\end{enumerate}
Then there exists a unique positive Radon measure $\mu$ on $X$ such
that $F(f)=\int_{X}f\dd\mu$ for all $f\in C_{c}(X)$.
\end{thm}

For the proof one may consult any standard text on measure theory,
e.g. Section 7.1 of \cite{Folland1999}. 

In this paper we study geometric forms of the Riesz representation
theorem. We start with a well known result about convex bodies. We
quickly give the basic definitions here, and refer the reader to \cite{Schneider2013}
or \cite{Hug2020} for more information. We denote by $\KK^{n}$ the
class of all compact convex sets $K\subseteq\RR^{n}$. Given $K,L\in\KK^{n}$
their Minkowski addition is defined as 
\[
K+L=\left\{ x+y:\ x\in K,\ y\in L\right\} .
\]
 For $K\in\KK^{n}$ and $\lambda>0$ we set $\lambda\cdot K=\left\{ \lambda x:\ x\in K\right\} $.
The operations $+$ and $\cdot$ turn $\KK^{n}$ into a cone. Finally,
the support function of a convex body $K\in\KK^{n}$ is the function
$h_{K}:\SS^{n-1}\to\RR$ defined by 
\[
h_{K}(\theta)=\max_{x\in K}\left\langle x,\theta\right\rangle .
\]
 Here $\left\langle \cdot,\cdot\right\rangle $ denotes the Euclidean
inner product on $\RR^{n}$ and $\SS^{n-1}=\left\{ x\in\RR^{n}:\ \left|x\right|=1\right\} $
denotes the unit sphere. 

We can now state a Riesz-type theorem for convex bodies:
\begin{thm}
\label{thm:riesz-bodies}Let $F:\KK^{n}\to\RR$ be a functional such
that:
\begin{enumerate}
\item $F$ is linear: $F(\alpha K+\beta L)=\alpha F(K)+\beta F(L)$ for
all $K,L\in\KK^{n}$ and $\alpha,\beta>0$.
\item $F$ is increasing: If $K,L\in\KK^{n}$ and $K\supseteq L$ then $F(K)\ge F(L)$.
\end{enumerate}
Then there exists a unique positive and finite Borel measure $\mu$
on $\SS^{n-1}$ such that $F(K)=\int_{\SS^{n-1}}h_{K}\dd\mu$.
\end{thm}

Recall that a finite Borel measure on $\SS^{n-1}$ or $\RR^{n}$ is
automatically Radon (see e.g. Theorems 1.1 and 1.3 of \cite{Billingsley1999}),
so we can ignore issues of regularity in this theorem and throughout
this paper. 

In its stated form Theorem \ref{thm:riesz-bodies} appears to be folklore.
However, Alexandrov knew and used this result for a specific function
$F$, and his proof works in complete generality (see \cite{Alexandrov1937}
for the original in Russian, and Section 4 of \cite{Alexandrov2019}
for an English translation). A modern presentation of Alexandrov's
result with essentially the same reasoning can be found as Theorem
4.1 of \cite{Hug2020}, and a few more references and historical remarks
can be found in Note 1 after Section 5.1 of \cite{Schneider2013}. 

Since we will use Theorem \ref{thm:riesz-bodies} in the sequel we
sketch its proof:
\begin{proof}[Proof Sketch.]
Let $C\left(\SS^{n-1}\right)$ denote the space of continuous functions
on the sphere with its usual supremum norm, and define 
\[
E=\left\{ h_{K}-h_{L}:\ K,L\in\KK^{n}\right\} \subseteq C\left(\SS^{n-1}\right).
\]
 It is well known that $E$ contains $C^{2}\left(\SS^{n-1}\right)$,
the space of twice differentiable functions, so in particular $E$
is dense in $C\left(\SS^{n-1}\right)$. We now define $\widetilde{F}:E\to\RR$
by 
\[
\widetilde{F}\left(h_{K}-h_{L}\right)=F(K)-F(L).
\]
 It is easy to check that $\widetilde{F}$ is well-defined, linear
and increasing. It follows that $\widetilde{F}$ is continuous, so
it has a unique extension to $C\left(\SS^{n-1}\right)$ which is again
linear and increasing. By Theorem \ref{thm:riesz-classic} there exists
a positive Random measure $\mu$ on $\SS^{n-1}$ such that $\widetilde{F}(f)=\int_{\SS^{n-1}}f\dd\mu$
for all $f\in C\left(\SS^{n-1}\right)$. In particular 
\[
F(K)=\widetilde{F}\left(h_{K}\right)=\int_{\SS^{n-1}}f\dd\mu.
\]
 Since $\SS^{n-1}$ is compact $\mu$ is clearly finite. 

Finally, for uniqueness, assume $\int_{\SS^{n-1}}h_{K}\dd\mu=\int_{\SS^{n-1}}h_{K}\dd\widetilde{\mu}$
for all $K\in\KK^{n}$. Then $\int_{\SS^{n-1}}f\dd\mu=\int_{\SS^{n-1}}f\dd\widetilde{\mu}$
for every $f\in E$, and therefore for every $f\in C\left(\SS^{n-1}\right)$.
By the uniqueness part of Theorem \ref{thm:riesz-classic} it follows
that $\mu=\widetilde{\mu}$.
\end{proof}
While interesting in its own right, Theorem \ref{thm:riesz-bodies}
also has applications to convex geometry. To explain the idea, assume
$G:\KK^{n}\to\RR$ is some measure of the ``size'' of convex bodies.
We expect $G$ to be increasing, but not necessarily linear. However,
often one can linearize $G$: we fix $L\in\KK^{n}$, and define $F_{L}:\KK^{n}\to\RR$
by 
\[
F_{L}(K)=\lim_{t\to0^{+}}\frac{G(L+tK)-G(L)}{t}.
\]
 For several natural choices of $G$ the functions $F_{L}$ will be
well-defined and linear. Since they are clearly increasing we can
apply Theorem \ref{thm:riesz-bodies} and conclude that 
\begin{equation}
\lim_{t\to0^{+}}\frac{G(L+tK)-G(L)}{t}=\int_{\SS^{n-1}}h_{K}\dd\mu_{L}\label{eq:body-derivative}
\end{equation}
for a measure $\mu_{L}$ which depends on $G$ and $L$. Studying
the measures $\mu_{L}$ can be of great importance.

The simplest possible example is the choice $G(K)=\left|K\right|$,
i.e. the volume of $K$. In this case we have the formula 
\[
\lim_{t\to0^{+}}\frac{\left|L+tK\right|-\left|L\right|}{t}=\int_{\SS^{n-1}}h_{K}\dd S_{L},
\]
 where $S_{L}$ is known as the \emph{surface area measure} of the
body $L$. 

This example can be extended using the theory of mixed volumes. For
example, we can take $G(K)=v_{i}(K)$, where $v_{i}$ is the $i$'th
intrinsic volume, or more generally 
\[
G(K)=V\left(L_{1},L_{2},\ldots,L_{n-m},\underbrace{K,K,\ldots,K}_{m\text{ times}}\right)
\]
 for some fixed convex bodies $L_{1},\ldots,L_{n-m}\in\KK^{n}$. In
these examples all linearizations will always be of the form $c\cdot F_{L_{1},\ldots,L_{n-1}}(K)=c\cdot V(L_{1},\ldots,L_{n-1},K)$
for some constant $c>0$ and some convex bodies $L_{1},\ldots,L_{n-1}$.
The measure which represents the functional $F_{L_{1},\ldots,L_{n-1}}$
is denoted by $S_{L_{1},\ldots,L_{n-1}}$ and is known as a \emph{mixed
area measure}. This example of $F_{L_{1},\ldots,L_{n-1}}$ is exactly
the one studied by Alexandrov in \cite{Alexandrov1937}. More information
on mixed volumes and mixed area measures can be found in \cite{Schneider2013}
or \cite{Hug2020}, but we will not need these notions for the rest
of the paper. 

As another possible extension one can take $G(K)=\nu(K)$ where $\nu$
is a measure on $\RR^{n}$ with a continuous density. Livshyts proved
in \cite{Livshyts2019} that in this case \eqref{eq:body-derivative}
still holds, and gave an explicit formula for the measure $\mu_{L}$. 

We now turn our attention from bodies to functions. A function $f:\RR^{n}\to[0,\infty)$
which is not identically $0$ is called log-concave if for every $x,y\in\RR^{n}$
and every $0\le\lambda\le1$ one has 
\[
f\left((1-\lambda)x+\lambda y\right)\ge f(x)^{1-\lambda}f(y)^{\lambda}.
\]
 In other words $f$ is log-concave if it is of the form $f=e^{-\phi}$
where $\phi:\RR^{n}\to(-\infty,\infty]$ is a convex function. We
denote by $\lc_{n}$ the class of upper semi-continuous log-concave
functions on $\RR^{n}$. Similarly we denote by $\cvx_{n}$ the class
of all lower semi-continuous convex functions on $\RR^{n}$ (which
are not identically $+\infty$). If $K\in\KK^{n}$ then $\oo_{K}\in\lc_{n}$
, where $\oo_{K}$ is the indicator function of $K$. In this sense
we have a natural embedding $\KK^{n}\hookrightarrow\lc_{n}$. 

It is well understood nowadays that even if one is ultimately only
interested in convex bodies, it is extremely useful to also consider
log-concave functions and treat them as ``generalized convex bodies''.
This opens the door to the use of various analytic and probabilistic
techniques in convex geometry, and allows us to make progress on previously
impenetrable problems. The systematic geometric treatment of log-concave
functions originated in the work of Klartag and Milman (\cite{Klartag2005})
and in the proof of the functional Santaló inequality by Artstein-Avidan,
Klartag and Milman (\cite{Artstein-Avidan2004}), even though in retrospect
one can interpret older results in this language. For example, the
Prékopa–Leindler inequality (\cite{Prekopa1971}, \cite{Leindler1972})
from the 70's can be interpreted as a functional version of the Brunn–Minkowski
inequality. 

Some of the earlier developments in this direction can be found in
Section 9.5 of \cite{Schneider2013} and in the survey \cite{Milman2008}.
The explosion in the field since then makes it impossible to include
here a manageable list of references. Instead, let us just mention
that studying \emph{functionals} on log-concave (or convex) functions,
i.e. maps $F:\lc_{n}\to\RR$, is a very active field of research.
Usually one studies valuations on this space, i.e. functionals $F$
which satisfy
\[
F\left(\max(f,g)\right)+F\left(\min(f,g)\right)=F(f)+F(g)
\]
 whenever $f,g,\max(f,g)\in\lc_{n}$. Many results on such valuations
including partial classifications were recently found by Colesanti,
Ludwig and Mussnig (\cite{Colesanti2017a,Colesanti2019,Colesanti2020,Colesanti2020a}),
by Mussnig (\cite{Mussnig2019,Mussnig2021}), by Alesker (\cite{Alesker2019}),
and by Knoerr (\cite{Knoerr2020,Knoerr2021}).

In this paper we are interested in Riesz-type representation theorems
on $\lc_{n}$. In other words, we will classify functionals $F:\lc_{n}\to\RR$
which are linear and increasing. Of course, this requires an addition
operation on $\lc_{n}$. The standard choice that appeared already
in \cite{Klartag2005} is the sup-convolution (also known as the Asplund
sum), defined by 
\[
\left(f\star g\right)(x)=\sup_{y\in\RR^{n}}\left(f(y)g(x-y)\right).
\]
Additionally, if $\lambda>0$ we define the dilation $\lambda\cdot f$
by $\left(\lambda\cdot f\right)(x)=f\left(\frac{x}{\lambda}\right)^{\lambda}$.
These operations extend the standard operations on convex bodies under
the embedding $\KK^{n}\hookrightarrow\lc_{n}$ mentioned above. 

A technical issue is that for $f,g\in\lc_{n}$ the sup-convolution
$f\star g$ could be equal to $+\infty$, which we do not allow, or
may fail to be upper semi-continuous. The latter problem can be fixed
by defining $f\star g$ to be the closure of the sup-convolution (see
e.g. Section 7 of \cite{Rockafellar1970}), but the first problem
doesn't have such a solution. These problems disappear for ``nice
enough'' functions as we shall soon see, so in practice this will
not cause difficulties anywhere in this paper. We just have to be
slightly careful and define a functional $F:\lc_{n}\to\RR$ to be
linear if 
\[
F\left(\left(\alpha\cdot f\right)\star\left(\beta\cdot g\right)\right)=\alpha F(f)+\beta F(g)
\]
 holds for every $f,g\in\lc_{n}$ and every $\alpha,\beta>0$ such
that $\left(\alpha\cdot f\right)\star\left(\beta\cdot g\right)\in\lc_{n}$. 

Another way to understand the addition operation on $\lc_{n}$ is
by using the \emph{support function}. For a function $f\in\lc_{n}$
we define its support function $h_{f}\in\cvx_{n}$ by $h_{f}=\left(-\log f\right)^{\ast}$.
Here $\ast$ denotes the classical Legendre transform, i.e. 
\[
\phi^{\ast}(y)=\sup_{x\in\RR^{n}}\left(\left\langle x,y\right\rangle -\phi(x)\right).
\]
 This definition also extends the classical one for convex bodies,
in the sense that $h_{\oo_{K}}=h_{K}$. The important thing for us
is that for every $f,g\in\lc_{n}$ and $\alpha,\beta>0$ we have $h_{\left(\alpha\cdot f\right)\star\left(\beta\cdot g\right)}=\alpha h_{f}+\beta h_{g}$
as expected. In fact, this property essentially characterizes the
support map (see Theorem 6 of \cite{Artstein-Avidan2010a}). Since
$\phi^{\ast\ast}=\phi$ for every $\phi\in\cvx_{n}$ we see that for
every $f\in\lc_{n}$ we have $f=e^{-h_{f}^{\ast}}$ and that every
$\phi\in\cvx_{n}$ is the support function of a unique $f\in\lc_{n}$. 

Since we now have an addition, we understand what it means for $F:\lc_{n}^{c}\to\RR$
to be linear. However, our first attempt at classifying linear and
increasing functionals will be very underwhelming:
\begin{thm}
\label{thm:riesz-functions-naive}Let $F:\lc_{n}\to\RR$ be a linear
and increasing functional. Then $F(f)=0$ for all $f\in\lc_{n}$. 
\end{thm}

In fact, the same theorem is true without the assumption that $F$
is increasing! The reason for such a disappointing result is that
the class of functions we are considering is too large. In Theorem
\ref{thm:riesz-bodies} we did not work with all closed convex sets,
but only with compact sets. In the same way we need to impose some
kind of ``compactness'' criterion on our log-concave functions.
One natural attempt is to work with coercive functions:
\begin{defn}
A convex function $\phi:\RR^{n}\to(-\infty,\infty]$ is called \emph{coercive}
if $\lim_{\left|x\right|\to\infty}\phi(x)=\infty$. A log-concave
function $f=e^{-\phi}$ is called coercive if $\phi$ is coercive.
We denote the classes of coercive convex and log-concave functions
by $\cvx_{n}^{c}$ and $\lc_{n}^{c}$ respectively. 
\end{defn}

Coercive log-concave functions are very well-behaved. For example,
every $f\in\lc_{n}^{c}$ satisfies $\int f<\infty$, and in fact decays
exponentially fast as $\left|x\right|\to\infty$. Moreover, for $f,g\in\lc_{n}^{c}$
the sup-convolution $f\star g$ is also in $\lc_{n}^{c}$, and in
particular is always finite and upper semi-continuous (see e.g. Lemma
2.3 of \cite{Hofstatter2021}). However, it turns out that the situation
for $\lc_{n}^{c}$ is not much better than the situation for $\lc_{n}$:
\begin{thm}
\label{thm:riesz-functions-c}Let $F:\lc_{n}^{c}\to\RR$ be a linear
and increasing functional. Then there exists $c\ge0$ such that $F(f)=c\cdot h_{f}(0)$
for all $f\in\lc_{n}$. 
\end{thm}

Instead, the correct class of functions one should consider is the
following:
\begin{defn}
A convex function $\phi:\RR^{n}\to(-\infty,\infty]$ is called \emph{super-coercive}
if 
\[
\lim_{\left|x\right|\to\infty}\frac{\phi(x)}{\left|x\right|}=\infty.
\]
A log-concave function $f=e^{-\phi}$ is called super-coercive if
$\phi$ is super-coercive. We denote the classes of super-coercive
convex and log-concave functions by $\cvx_{n}^{sc}$ and $\lc_{n}^{sc}$
respectively. 
\end{defn}

To understand in which sense functions $f\in\lc_{n}^{sc}$ are ``compact'',
note that a closed convex set $K\subseteq\RR^{n}$ is compact if and
only if $h_{K}<\infty$ everywhere on $\SS^{n-1}$. In the same way,
it is well known and not difficult to prove that $f\in\lc_{n}$ is
super-coercive if and only if $h_{f}<\infty$ everywhere on $\RR^{n}$.
This also shows that $\lc_{n}^{sc}$ is closed under sup-convolution.
The class of super-coercive convex functions is the one used by Colesanti,
Ludwig and Mussnig in \cite{Colesanti2020a} to prove a Hadwiger type
theorem. We will now see that on this class the Riesz representation
theorem is much more interesting:
\begin{thm}
\label{thm:riesz-functions-sc}Let $F:\lc_{n}^{sc}\to\RR$ be a linear
and increasing functional. Then there exists a unique positive and
finite Borel measure $\mu$ on $\RR^{n}$ with compact support such
that $F(f)=\int_{\RR^{n}}h_{f}\dd\mu$.
\end{thm}

Theorem \ref{thm:riesz-functions-sc} seems to be the exact analogue
of Theorem \ref{thm:riesz-bodies}, and the best one could hope for.
However, as we will see in Section \ref{sec:surface-area} there are
useful applications where we want to allow $F$ to attain the value
$+\infty$ . In this case one can prove an interesting Riesz-type
theorem on the entire family $\lc_{n}$. In fact, one obtains not
only the family of functionals from Theorem \ref{thm:riesz-functions-sc}
but also a new surprising family of functionals: 
\begin{thm}
\label{thm:riesz-functions-general}Let $F:\lc_{n}\to(-\infty,\infty]$
be a linear and increasing functional. Assume further that:
\begin{enumerate}
\item There exists a function $f_{0}\in\lc_{n}$ such that $\int f_{0}>0$
and $F(f_{0})<\infty$. 
\item For every sequence $\left\{ f_{i}\right\} _{i=1}^{\infty}\subseteq\lc_{n}$
such that $f_{i}\uparrow f\in\lc_{n}$ we have $F(f_{i})\to F(f)$. 
\end{enumerate}
Then there exists a unique positive and finite Borel measure $\mu$
on $\RR^{n}$ with a finite first moment, and a unique positive and
finite Borel measure $\nu$ on $\SS^{n-1}$, such that
\begin{equation}
F(f)=\int_{\RR^{n}}h_{f}\dd\mu+\int_{\SS^{n-1}}h_{K_{f}}\dd\nu.\label{eq:riesz-func-repr}
\end{equation}
 Conversely, every functional of the form \eqref{eq:riesz-func-repr}
satisfies the assumptions of the theorem. 
\end{thm}

To explain the notation used in the theorem, $K_{f}=\overline{\left\{ x\in\RR^{n}:f(x)>0\right\} }$
is the support of $f$. This is a closed convex set, so one may indeed
consider its support function $h_{K_{f}}:\SS^{n-1}\to(-\infty,\infty]$.
The notation $f_{i}\uparrow f$ simply means that $f_{1}\le f_{2}\le f_{3}\le\cdots$
and $\lim_{i\to\infty}f_{i}(x)=f(x)$ for all $x$. 

The extra assumptions in Theorem \ref{thm:riesz-functions-general}
are not just an artifact of our proof. We will see how removing any
of the two assumptions creates more linear and increasing functionals
that are usually not so interesting to consider. In practice it is
usually easy to verify that these extra conditions are satisfied in
specific applications. As will be clear from the proof, the same theorem
holds if $\lc_{n}$ is replaced with $\lc_{n}^{c}$ or $\lc_{n}^{sc}$. 

The rest of this paper is organized as follows: First, in Section
\ref{sec:degenerate} we give the short proofs of Theorems \ref{thm:riesz-functions-naive}
and \ref{thm:riesz-functions-c}. This can be done directly, without
appealing to the classical Riesz theorem. In Section \ref{sec:super-coercive}
we prove Theorem \ref{thm:riesz-functions-sc}. In Section \ref{sec:cvx-infinity}
we prove some simple results above the behavior of functions $\phi\in\cvx_{n}$
``at infinity''. These results will be needed for the proof of Theorem
\ref{thm:riesz-functions-general} in Section \ref{sec:non-finite}.
Finally in Section \ref{sec:surface-area} we connect our main theorems
to recent works about the surface area measures of log-concave functions.
We show that these works provide important examples of linear functionals
on $\lc_{n}$, and give an application of our results to this theory. 

\section{\label{sec:degenerate}The degenerate cases}

Each of our main theorems (Theorems \ref{thm:riesz-functions-naive},
\ref{thm:riesz-functions-c}, \ref{thm:riesz-functions-sc} and \ref{thm:riesz-functions-general})
can be reformulated in the language of convex functions. For example,
Theorem \ref{thm:riesz-functions-naive} is equivalent to the following:

\begin{manualtheorem}{\ref*{thm:riesz-functions-naive}$^\star$}\label{thm:conv-naive}

Let $F:\cvx_{n}\to\RR$ be linear and increasing. Then $F(\phi)=0$
for all $\phi\in\cvx_{n}$. 

\end{manualtheorem}

Here and everywhere else in the paper, a functional $F$ defined on
a domain $D\subseteq\cvx_{n}$ is called linear if it is linear with
respect to the usual pointwise addition: $F(\alpha\phi+\beta\psi)=\alpha F(\phi)+\beta F(\psi)$
for all $\phi,\psi\in D$ and all $\alpha,\beta>0$ such that $\alpha\phi+\beta\psi\in D$.

We now explain this equivalence by first proving Theorem \ref{thm:conv-naive}
and then using it to prove Theorem \ref{thm:riesz-functions-naive}:
\begin{proof}[Proof of Theorem \ref{thm:conv-naive}]
Write $c=F(\oo)$ where $\oo$ denotes the constant function. For
every $\phi\in\cvx_{n}$ and every $p\in\RR^{n}$ such that $\phi(p)<\infty$
we have 
\[
\phi+\oo_{\left\{ p\right\} }^{\infty}=\phi(p)\cdot\oo+\oo_{\left\{ p\right\} }^{\infty},
\]
 where 
\[
\oo_{K}^{\infty}(x)=\begin{cases}
0 & x\in K\\
+\infty & \text{Otherwise}
\end{cases}
\]
 denotes the convex indicator function. By linearity of $F$ we then
have 
\[
F(\phi)+F\left(\oo_{\left\{ p\right\} }^{\infty}\right)=\phi(p)\cdot F(\oo)+F\left(\oo_{\left\{ p\right\} }^{\infty}\right),
\]
 or $F(\phi)=c\cdot\phi(p)$.

If we now take for example $\phi(x)=\left|x\right|$ and fix $\theta\in\SS^{n-1}$
we get that $c=c\cdot\phi\left(\theta\right)=F(\phi)=c\cdot\phi(2\theta)=2c$.
Hence $c=0$, so $F(\phi)=0$ for all $\phi\in\cvx_{n}$. 
\end{proof}
As was mentioned in the introduction, we actually never use in the
proof the fact that $F$ is increasing. Therefore there are no non-trivial
linear functions $F:\cvx_{n}\to\RR$. 
\begin{proof}[Proof of Theorem \ref{thm:riesz-functions-naive}]
Assume $F:\lc_{n}\to\RR$ is linear and increasing. Define $G:\cvx_{n}\to\RR$
by $G(\phi)=F\left(e^{-\phi^{\ast}}\right)$. Since $F$ is increasing
and the Legendre transform is order reversing, $G$ is also increasing.
Moreover 
\begin{align*}
G(\alpha\phi+\beta\psi) & =F\left(e^{-\left(\alpha\phi+\beta\psi\right)^{\ast}}\right)=F\left(\left(\alpha\cdot e^{-\phi}\right)\star\left(\beta\cdot e^{-\psi}\right)\right)\\
 & =\alpha F(e^{-\phi})+\beta F(e^{-\psi})=\alpha G(\phi)+\beta G(\psi)
\end{align*}
 so $G$ is linear (to avoid confusion, recall that the notation $\cdot$
in expressions like $\alpha\cdot e^{-\phi}$ does not refer to the
pointwise multiplication but to the dilation defined in the introduction).

From Theorem \ref{thm:conv-naive} we deduce that $G\equiv0$. Therefore
for every $f\in\lc_{n}$ we have $F(f)=F\left(e^{-h_{f}^{\ast}}\right)=G(h_{f})=0$. 
\end{proof}
For Theorem \ref{thm:riesz-functions-c} the argument is similar.
The set of support functions $\left\{ h_{f}:\ f\in\lc_{n}^{c}\right\} $
is easily seen to be 
\[
\widetilde{\cvx}_{n}=\left\{ \phi\in\cvx_{n}:\ \phi\text{ is finite in a neighborhood of }0\right\} .
\]

Therefore Theorem \ref{thm:riesz-functions-c} will be equivalent
to the following:

\begin{manualtheorem}{\ref*{thm:riesz-functions-c}$^\star$}\label{thm:conv-c}

Let $F:\widetilde{\cvx}_{n}\to\RR$ be linear and increasing. Then
there exists $c\ge0$ such that $F(\phi)=c\cdot\phi(0)$ for all $\phi\in\widetilde{\cvx}_{n}$. 

\end{manualtheorem}
\begin{proof}
Write $c=F(\oo)$, and fix $\phi\in\widetilde{\cvx}_{n}$. Since $\phi$
is convex and finite in a neighborhood of $0$ it is continuous at
$0$. Therefore given $\epsilon>0$ there exists $r>0$ such that
$\left|\phi(x)-\phi(0)\right|<\epsilon$ for $x\in rB_{2}^{n}$, the
Euclidean Ball of radius $r$ centered at $0$. It follows that 
\[
\left(\phi(0)-\epsilon\right)\cdot\oo+\oo_{rB_{2}^{n}}^{\infty}\le\phi+\oo_{rB_{2}^{n}}^{\infty}\le\left(\phi(0)+\epsilon\right)\cdot\oo+\oo_{rB_{2}^{n}}^{\infty}.
\]
 Using the linearity and monotonicity of $F$ we see that 
\[
\left(\phi(0)-\epsilon\right)F(\oo)+F\left(\oo_{rB_{2}^{n}}^{\infty}\right)\le F(\phi)+F\left(\oo_{rB_{2}^{n}}^{\infty}\right)\le\left(\phi(0)+\epsilon\right)F(\oo)+F\left(\oo_{rB_{2}^{n}}^{\infty}\right),
\]
 or $\left|F\left(\phi\right)-c\phi(0)\right|\le c\epsilon$. Since
this is true for all $\epsilon>0$ we conclude that $F(\phi)=c\cdot\phi(0)$
as we wanted.
\end{proof}
Theorem \ref{thm:riesz-functions-c} follows from \ref{thm:conv-c}
in exactly the same way Theorem \ref{thm:riesz-functions-naive} follows
from \ref{thm:conv-naive}, so we will not repeat the argument. 

\section{\label{sec:super-coercive}A representation theorem for super-coercive
functions}

In this section we prove Theorem \ref{thm:riesz-functions-sc}. As
was explained in Section \ref{sec:introduction}, the fact that $f$
is super-coercive is equivalent to $h_{f}$ being everywhere finite.
Hence we define:
\begin{defn}
We denote by $\cvx_{n}^{F}$ the class of all convex functions $\phi\in\cvx_{n}$
such that $\phi(x)<\infty$ for all $x\in\RR^{n}$.
\end{defn}

Just like in Section \ref{sec:degenerate}, Theorem \ref{thm:riesz-functions-sc}
is an immediate corollary of the following Riesz type theorem for
$\cvx_{n}^{F}$:

\begin{manualtheorem}{\ref*{thm:riesz-functions-sc}$^\star$}\label{thm:riesz-cvxf}

Let $F:\cvx_{n}^{F}\to\RR$ be linear and increasing. Then there exists
a unique positive and finite Borel measure $\mu$ on $\RR^{n}$ with
compact support such that $F(\phi)=\int_{\RR^{n}}\phi\dd\mu$.

\end{manualtheorem}

Towards the proof we consider the space of functions 
\[
E=\left\{ \phi+f:\ \phi\in\cvx_{n}^{F}\text{ and }f\in C_{c}\left(\RR^{n}\right)\right\} .
\]
The first main step in the proof of Theorem \ref{thm:riesz-cvxf}
is the following result which extends functionals from $\cvx_{n}^{F}$
to $E$: 
\begin{lem}
\label{lem:E-extend}Let $F:\cvx_{n}^{F}\to\RR$ be a linear and increasing
functional. Then $F$ can be extended to a functional $F:E\to\RR$
which is again linear and increasing.
\end{lem}

\begin{proof}
Our first step is to extend $F$ to the smaller space 
\[
\widetilde{E}=\left\{ \phi+f:\ \phi\in\cvx_{n}^{F}\text{ and }f\in C_{c}^{2}\left(\RR^{n}\right)\right\} ,
\]
 where $C_{c}^{2}(\RR^{n})$ denotes the $C^{2}$-smooth compactly
supported functions. 

Towards this goal we define $\rho_{a}\in\cvx_{n}^{F}$ for every $a>0$
by 
\[
\rho_{a}(x)=\begin{cases}
a\frac{\left|x\right|^{2}}{2} & \left|x\right|\le a\\
a^{2}\left|x\right|-\frac{a^{3}}{2} & \left|x\right|\ge a.
\end{cases}
\]
We claim that if $f\in C_{c}^{2}(\RR^{n})$ then $f+\rho_{a}\in\cvx_{n}^{F}$
for large enough $a$. Indeed, assume $f$ is supported on $B(0,r)$,
an open ball of radius $r>0$ around the origin. Since the Hessian
$\nabla^{2}f$ is continuous and compactly supported there exists
$m>0$ such that $\nabla^{2}f\succeq-m\cdot Id$ in the sense of positive
definite matrices. Choose $a>\max(r,m)+1$. Then for every $x\in B(0,a)$
we have 
\[
\nabla^{2}\left(f+\rho_{a}\right)(x)=\nabla^{2}f(x)+\nabla^{2}\rho_{a}(x)=\nabla^{2}f(x)+a\cdot Id\succeq(a-m)\cdot Id\succeq Id,
\]
 so $f+\rho_{a}$ is convex in a neighborhood of $x$. If on the other
hand $x\notin B(0,r+\frac{1}{2})$ then $f+\rho_{a}=\rho_{a}$ in
a neighborhood of $x$. It follows that $f+\rho_{a}$ is convex in
a neighborhood of every point of $\RR^{n}$, so it is convex.

We now extend $F$ to $\widetilde{E}$ by setting 
\[
F(\phi+f)=F(\phi)+F\left(f+\rho_{a}\right)-F(\rho_{a})
\]
 for some $a>0$ such that $f+\rho_{a}\in\cvx_{n}^{F}$. To see that
this is well-defined, fix $\phi_{1},\phi_{2}\in\cvx_{n}^{F}$, $f_{1},f_{2}\in C_{c}^{2}(\RR^{n})$
and $a,b>0$ such that $\phi_{1}+f_{1}=\phi_{2}+f_{2}$ and $f_{1}+\rho_{a},f_{2}+\rho_{b}\in\cvx_{n}^{F}$.
Then by linearity of $F$ on $\cvx_{n}^{F}$ we have 
\begin{align*}
F\left(\phi_{1}\right)+F\left(f_{1}+\rho_{a}\right)+F(\rho_{b}) & =F\left(\phi_{1}+f_{1}+\rho_{a}+\rho_{b}\right)=F(\phi_{2}+f_{2}+\rho_{b}+\rho_{a})\\
 & =F(\phi_{2})+F(f_{2}+\rho_{b})+F(\rho_{a}),
\end{align*}
 so indeed 
\[
F\left(\phi_{1}\right)+F\left(f_{1}+\rho_{a}\right)-F(\rho_{a})=F(\phi_{2})+F(f_{2}+\rho_{b})-F(\rho_{b}).
\]

Next, to show that $F$ is linear on $\widetilde{E}$, fix $\phi_{1},\phi_{2}\in\cvx_{n}^{F}$
, $f_{1},f_{2}\in C_{c}^{2}(\RR^{n})$ and $\alpha,\beta>0$. Choose
$a>0$ such that $f_{1}+\rho_{a},f_{2}+\rho_{a},\alpha f_{1}+\beta f_{2}+\rho_{a}\in\cvx_{n}^{F}$.
Then using the linearity of $F$ on $\cvx_{n}^{F}$ we can compute:
\begin{align*}
F\left(\alpha(\phi_{1}+f_{1})+\beta\left(\phi_{2}+f_{2}\right)\right) & =F\left(\left(\alpha\phi_{1}+\beta\phi_{2}\right)+\left(\alpha f_{1}+\beta f_{2}\right)\right)\\
 & =F(\alpha\phi_{1}+\beta\phi_{2})+F(\alpha f_{1}+\beta f_{2}+\rho_{a})-F(\rho_{a})\\
 & =F(\alpha\phi_{1}+\beta\phi_{2})+F(\alpha f_{1}+\beta f_{2}+\rho_{a}+\alpha\rho_{a}+\beta\rho_{a})-(\alpha+\beta+1)F(\rho_{a})\\
 & =F(\alpha\phi_{1}+\beta\phi_{2})+F\left(\alpha\left(f_{1}+\rho_{a}\right)+\beta\left(f_{2}+\rho_{a}\right)\right)-(\alpha+\beta)F(\rho_{a})\\
 & =\alpha\left(F(\phi_{1})+F(f_{1}+\rho_{a})-F(\rho_{a})\right)+\beta\left(F(\phi_{2})+F(f_{2}+\rho_{a})-F(\rho_{a})\right)\\
 & =\alpha F(\phi_{1}+f_{1})+\beta F(\phi_{2}+f_{2}),
\end{align*}
 which shows that $F$ is linear on $\widetilde{E}$.

Similarly we show that $F$ is increasing on $\widetilde{E}$: If
$\phi_{1}+f_{1}\le\phi_{2}+f_{2}$ we can choose $a>0$ such that
$f_{1}+\rho_{a},f_{2}+\rho_{a}\in\cvx_{n}^{F}$ and then 
\begin{align*}
F(\phi_{1}+f_{1}) & =F(\phi_{1})+F(f_{1}+\rho_{a})-F(\rho_{a})=F(\phi_{1}+f_{1}+\rho_{a})-F(\rho_{a})\\
 & \le F(\phi_{2}+f_{2}+\rho_{a})-F(\rho_{a})=F(\phi_{2}+f_{2}).
\end{align*}

Our next step is to extend $F$ from $\widetilde{E}$ to $E$. Since
$F$ is linear and increasing it is also continuous with respect to
the supremum norm, in the sense that
\[
\left|F(\phi_{1}+f_{1})-F(\phi_{2}+f_{2})\right|\le\left\Vert \left(\phi_{1}+f_{1}\right)-\left(\phi_{2}+f_{2}\right)\right\Vert _{\infty}\cdot F(\oo).
\]
Here $\oo\in\cvx_{n}^{F}$ denotes the constant function. Of course
the right hand side may be equal to $+\infty$, in which case the
claim is trivial. Since $\widetilde{E}$ is dense in $E$ it follows
that $F$ can be uniquely extended to a continuous linear functional
on $E$. 

It only remains to show that $F$ is increasing on $E$. To this end
note that for every $f\in C_{c}\left(\RR^{n}\right)$ one can find
a sequence $\left\{ g_{i}\right\} _{i=1}^{\infty}\subseteq C_{c}^{2}(\RR^{n})$
such that $g_{i}\to f$ uniformly and $g_{i}\ge f$ for all $i$ (or
$g_{i}\le f$ for all $i$). Assume now that $\phi_{1}+f_{1},\phi_{2}+f_{2}\in E$
and $\phi_{1}+f_{1}\le\phi_{2}+f_{2}$. Choose a sequence $\left\{ g_{i}\right\} _{i=1}^{\infty}\subseteq C_{c}^{2}(\RR^{n})$
approximating $f_{1}$ from below and a sequence $\left\{ h_{i}\right\} _{i=1}^{\infty}\subseteq C_{c}^{2}(\RR^{n})$
approximating $f_{2}$ from above. Then 
\[
F(\phi_{1}+f_{1})=\lim_{i\to\infty}F(\phi_{1}+g_{i})\le\lim_{i\to\infty}F(\phi_{2}+h_{i})=F(\phi_{2}+f_{2}),
\]
finishing the proof. 
\end{proof}
The reader may wonder about the choice of the functions $\rho_{a}$
in the proof above. It appears that a simpler choice such as $\rho_{a}(x)=a\frac{\left|x\right|^{2}}{2}$
would work just as well and slightly simplify the proof. This is correct,
but in Section \ref{sec:non-finite} we will claim that the proof
above can also serve as proof of Lemma \ref{lem:El-extend}, and there
such simpler choices will not be possible.

To proceed we will also need the following lemma about fast growing
convex function. Similar statements have undoubtedly appeared in the
literature before, but as we were unable to find a suitable reference
we provide the proof: 
\begin{lem}
\label{lem:fast-growing}
\begin{enumerate}
\item For every $\phi\in\cvx_{n}^{F}$ there exists $\psi\in\cvx_{n}^{F}$
such that $\psi\ge0$ and $\lim_{\left|x\right|\to\infty}\frac{\psi(x)}{\phi(x)}=+\infty$.
\item Let $\mu$ be a positive Borel measure on $\RR^{n}$ which is not
compactly supported. Then there exists $\phi\in\cvx_{n}^{F}$ with
$\int_{\RR^{n}}\phi\dd\mu=+\infty$. 
\end{enumerate}
\end{lem}

\begin{proof}
Both parts of the lemma rely on the same principle: If $\left\{ a_{k}\right\} _{k=1}^{\infty}$
is an arbitrary sequence of real numbers, then one can find an increasing
convex function $\rho:[0,\infty)\to[0,\infty)$ such that $\rho(k)>a_{k}$
for all $k\in\NN$.

To prove this principle define a sequence $\left\{ b_{k}\right\} _{k=0}^{\infty}$
inductively by setting $b_{0}=1$ and 
\[
b_{k+1}=\max\left\{ 2b_{k},a_{k+1}\right\} +1
\]
 for $k\ge1$. Clearly $b_{k}>\max\left\{ a_{k},0\right\} $ for all
$k\ge1$, and since
\[
\frac{b_{k+1}+b_{k-1}}{2}\ge\frac{b_{k+1}}{2}>b_{k}
\]
the sequence $\left\{ b_{k}\right\} _{k=0}^{\infty}$ is convex and
increasing. Define $\rho:[0,\infty)\to[0,\infty)$ by setting $\rho(k)=b_{k}$
and extending $\rho$ to be linear on any interval of the form $[k,k+1]$.
Then $\rho$ is the required function. 

We can now prove the two parts of the lemma:
\begin{enumerate}
\item Given $\phi\in\cvx_{n}^{F}$ we choose $\rho:[0,\infty)\to[0,\infty)$
to be convex and increasing and satisfy 
\[
\rho(k)\ge k\cdot\max\left\{ \phi(x):\ \left|x\right|\le k+1\right\} 
\]
 for all $k\in\NN$. Define $\psi\in\cvx_{n}^{F}$ by $\psi(x)=\rho\left(\left|x\right|\right)$.
Then for every $x\in\RR^{n}$ such that $k\le\left|x\right|\le k+1$
we have 
\[
\frac{\psi(x)}{\phi(x)}\ge\frac{\rho(k)}{\phi(x)}\ge\frac{k\cdot\phi(x)}{\phi(x)}=k,
\]
 so $\lim_{\left|x\right|\to\infty}\frac{\psi(x)}{\phi(x)}=+\infty$.
\item Given $\mu$ we set $a_{k}=\mu\left(\left\{ x\in\RR^{n}:\ k\le\left|x\right|<k+1\right\} \right)$
and $I=\left\{ k\in\NN:\ a_{k}>0\right\} $. Since $\mu$ is not compactly
supported, $I$ must be infinite. We choose $\rho:[0,\infty)\to[0,\infty)$
to satisfy 
\[
\rho(k)\ge\begin{cases}
\frac{1}{a_{k}} & \text{If }a_{k}>0\\
0 & \text{otherwise},
\end{cases}
\]
 and define $\phi\in\cvx_{n}^{F}$ by $\phi(x)=\rho\left(\left|x\right|\right)$.
Then 
\[
\int\phi\dd\mu\ge\sum_{k=1}^{\infty}\int_{k\le\left|x\right|<k+1}\phi\dd\mu\ge\sum_{k=1}^{\infty}a_{k}\rho(k)\ge\sum_{k\in I}1=+\infty,
\]
 finishing the proof. 
\end{enumerate}
\end{proof}
We can now prove Theorem \ref{thm:riesz-cvxf}:
\begin{proof}[Proof of Theorem \ref{thm:riesz-cvxf}]
Assume $F:\cvx_{n}^{F}\to\RR$ is linear and increasing. By Lemma
\ref{lem:E-extend} $F$ can be extended to a linear increasing functional
on $E$ which we also denote by $F$. In particular the restriction
$\left.F\right|_{C_{c}(\RR^{n})}$ satisfies the assumptions of the
classical Riesz theorem (Theorem \ref{thm:riesz-classic}). Hence
there exists a unique positive Radon measure $\mu$ on $\RR^{n}$
such that $F(f)=\int f\dd\mu$ for all $f\in C_{c}(\RR^{n})$. Our
main goal is to show that the same formula holds for every $\phi\in\cvx_{n}^{F}$.

Fix $\phi\in\cvx_{n}^{F}$ and assume first that $\phi\ge0$. For
every $R>0$ let $\eta_{R}\in C_{c}\left(\RR^{n}\right)$ be a function
such that $0\le\eta_{R}\le1$, $\eta_{R}\equiv1$ on the ball $B(0,R)$,
and $\eta_{R}$ is compactly supported. Since $\phi\ge\phi\eta_{R}$
we have 
\[
F(\phi)\ge F\left(\phi\eta_{R}\right)=\int_{\RR^{n}}\phi\eta_{R}\dd\mu\ge\int_{B(0,R)}\phi\dd\mu,
\]
 where the middle equality holds since $\phi\eta_{R}\in C_{c}\left(\RR^{n}\right)$.
Letting $R\to\infty$ and using the monotone convergence theorem we
see that $F(\phi)\ge\int\phi\dd\mu$.

For the reverse inequality, we use Lemma \ref{lem:fast-growing} and
choose $\psi\in\cvx_{n}^{F}$ such that $\psi\ge0$ and $\lim_{\left|x\right|\to\infty}\frac{\psi(x)}{\phi(x)}=+\infty$.
This implies that for every $R>0$ the function 
\[
\zeta_{R}(x)=\max\left\{ R\phi(x)-\psi(x),0\right\} 
\]
 is compactly supported. Since $R\phi-\psi\le\zeta_{R}\le R\phi$
we have $\phi\le\frac{1}{R}\left(\psi+\zeta_{R}\right)$ and both
sides are in $E$. By linearity and monotonicity of $F$ we have
\begin{align*}
F(\phi) & =F\left(\frac{\psi+\zeta_{R}}{R}\right)\le\frac{1}{R}F(\psi)+\frac{1}{R}F(\zeta_{R})=\frac{1}{R}F(\psi)+\frac{1}{R}\int\zeta_{R}\dd\mu\\
 & \le\frac{1}{R}F(\psi)+\frac{1}{R}\int R\phi\dd\mu=\frac{1}{R}F(\psi)+\int\phi\dd\mu.
\end{align*}
 Letting $R\to\infty$ we conclude that $F(\phi)\le\int\phi\dd\mu$,
finishing the proof in the case $\phi\ge0$. 

If $\phi\in\cvx_{n}^{F}$ is not necessarily positive, we use the
fact that for every $\phi\in\cvx_{n}^{F}$ there exists $\alpha,\beta>0$
such that $\phi(x)\ge-\alpha\left|x\right|-\beta$ for all $x$. Since
$\phi+\alpha\left|x\right|+\beta\ge0$ we have 
\begin{align*}
F(\phi) & =F(\phi+\alpha\left|x\right|+\beta)-F\left(\alpha\left|x\right|+\beta\right)\\
 & =\int\left(\phi+\alpha\left|x\right|+\beta\right)\dd\mu-\int\left(\alpha\left|x\right|+\beta\right)\dd\mu=\int\phi\dd\mu
\end{align*}
 finishing the proof that $F(\phi)=\int\phi\dd\mu$ for all $\phi\in\cvx_{n}$.

In particular we have $\int\dd\mu=F\left(\oo\right)<\infty$, so $\mu$
must be finite. Moreover, since $F(\phi)<\infty$ for all $\phi\in\cvx_{n}^{F}$
it follows from Lemma \ref{lem:fast-growing} that $\mu$ is compactly
supported. 

It only remains to show that $\mu$ is unique. Assume $\int\phi\dd\mu_{1}=\int\phi\dd\mu_{2}$
for all $\phi\in\cvx_{n}^{F}$. As we saw in the proof of Lemma \ref{lem:E-extend},
every $f\in C_{c}^{2}(\RR^{n})$ can be written as a difference $f=\phi_{1}-\phi_{2}$
for $\phi_{1},\phi_{2}\in\cvx_{n}^{F}$. Therefore $\int f\dd\mu_{1}=\int f\dd\mu_{2}$
for all $f\in C_{c}^{2}(\RR^{n})$, and by approximation the same
holds for every $f\in C_{c}(\RR^{n})$. It then follows from the classical
Riesz theorem that $\mu_{1}=\mu_{2}$. 
\end{proof}
The proof of Theorem \ref{thm:riesz-functions-sc} from Theorem \ref{thm:riesz-cvxf}
works in the usual way, so we will not repeat the argument.

\section{\label{sec:cvx-infinity}Behavior at infinity of convex functions}

We now turn our attention to Theorem \ref{thm:riesz-functions-general}.
In this section we collect some properties of convex functions that
we will need for the proof. As before, we want to restate the theorem
in the language of pointwise linear functionals on convex functions.
However, Theorem \ref{thm:riesz-functions-general} involves not only
$h_{f}$ but also $h_{K_{f}}$, where $K_{f}$ is the support of $f$.
For this reason we need to know how to recover $h_{K_{f}}$ from $h_{f}$:
\begin{defn}
Given $\phi\in\cvx_{n}$ we define $\bd{\phi}:\SS^{n-1}\to(-\infty,\infty]$
by 
\begin{equation}
\bd{\phi}(\theta)=\lim_{\lambda\to\infty}\frac{\phi(p+\lambda\theta)}{\lambda},\label{eq:convex-bd}
\end{equation}
 where $p\in\RR^{n}$ is an arbitrary point such that $\phi(p)<\infty$. 
\end{defn}

\begin{prop}
The limit in definition \eqref{eq:convex-bd} exists and is independent
of $p$. Moreover, for every $f\in\lc_{n}$ we have $h_{K_{f}}=\bd{h_{f}}$. 
\end{prop}

\begin{proof}
Given $\phi\in\cvx_{n}$ we define $f\in\lc_{n}$ by $f=e^{-\phi^{\ast}}$
so that $h_{f}=\phi$. Fix $\theta\in\SS^{n-1}$ and fix a point $x\in\RR^{n}$
with $f(x)>0$, which means that $\phi^{\ast}(x)<\infty$. Since $\phi=\phi^{\ast\ast}$
we conclude that for every $\lambda>0$ we have
\[
\left\langle x,\theta\right\rangle =\frac{\left(\left\langle x,p+\lambda\theta\right\rangle -\phi^{\ast}(x)\right)+\phi^{\ast}(x)-\left\langle x,p\right\rangle }{\lambda}\le\frac{\phi(p+\lambda\theta)+\phi^{\ast}(x)-\left\langle x,p\right\rangle }{\lambda},
\]
 and therefore 
\[
\left\langle x,\theta\right\rangle \le\liminf_{\lambda\to\infty}\frac{\phi(p+\lambda\theta)+\phi^{\ast}(x)-\left\langle x,p\right\rangle }{\lambda}=\liminf_{\lambda\to\infty}\frac{\phi(p+\lambda\theta)}{\lambda}
\]
 Hence we have 
\[
h_{K_{f}}(\theta)=\sup_{x:\ f(x)>0}\left\langle x,\theta\right\rangle \le\liminf_{\lambda\to\infty}\frac{\phi(p+\lambda\theta)}{\lambda}.
\]

Conversely, since $\phi(p)<\infty$ for every $y\in\RR^{n}$ we have
\[
\phi^{\ast}(y)=\sup_{x\in\RR^{n}}\left[\left\langle x,y\right\rangle -\phi(x)\right]\ge\left\langle p,y\right\rangle -\phi(p).
\]
 Hence for every $\lambda>0$ we obtain
\begin{align*}
\frac{\phi(p+\lambda\theta)}{\lambda} & =\frac{1}{\lambda}\cdot\sup_{y\in\RR^{n}}\left[\left\langle y,p+\lambda\theta\right\rangle -\phi^{\ast}(y)\right]=\frac{1}{\lambda}\cdot\sup_{y\in K_{f}}\left[\left\langle y,p+\lambda\theta\right\rangle -\phi^{\ast}(y)\right]\\
 & \le\frac{1}{\lambda}\sup_{y\in K_{f}}\left[\left\langle y,p+\lambda\theta\right\rangle -\left\langle p,y\right\rangle +\phi(p)\right]=\frac{\phi(p)}{\lambda}+h_{K_{f}}(\theta).
\end{align*}
Therefore $\limsup_{\lambda\to\infty}\frac{\phi(p+\lambda\theta)}{\lambda}\le h_{K_{f}}(\theta)$.
Together it follows that 
\[
h_{K_{f}}(\theta)=\lim_{\lambda\to\infty}\frac{\phi(p+\lambda\theta)}{\lambda},
\]
 which in particular shows that the limit exists and is independent
of $p\in\RR^{n}$. 
\end{proof}
We can now restate Theorem \ref{thm:riesz-functions-general} in the
language of convex functions:

\begin{manualtheorem}{\ref*{thm:riesz-functions-general}$^\star$}\label{thm:riesz-cvx-g}Let
$F:\cvx_{n}\to(-\infty,\infty]$ be a linear and increasing functional.
Assume further that:
\begin{enumerate}
\item \label{enu:cond-finitness}There exists a function $\phi_{0}\in\cvx_{n}$
with $\int e^{-\phi_{0}^{\ast}}>0$ and $F(\phi_{0})<\infty$. 
\item \label{enu:cond-cont}For every $\left\{ \phi_{i}\right\} _{i=1}^{\infty},\phi\in\cvx_{n}$
such that $\phi_{i}^{\ast}\downarrow\phi^{\ast}$ we have $F(\phi_{i})\to F(\phi)$. 
\end{enumerate}
Then there exists a unique finite Borel measure $\mu$ on $\RR^{n}$
with a finite first moment, and a unique finite Borel measure $\nu$
on $\SS^{n-1}$, such that 
\begin{equation}
F(\phi)=\int_{\RR^{n}}\phi\dd\mu+\int_{\SS^{n-1}}\bd{\phi}\dd\nu\label{eq:riesz-cvx-main}
\end{equation}
 for all $\phi\in\cvx_{n}$. Conversely, every functional of the form
\eqref{eq:riesz-cvx-main} satisfies the assumptions of the theorem.

\end{manualtheorem}

The equivalence of Theorems \ref{thm:riesz-functions-general} and
\ref{thm:riesz-cvx-g} is proved in the usual way, so we will not
repeat the argument again. Instead we will start by proving the ``conversely''
part of Theorem \ref{thm:riesz-cvx-g}. It is obvious that the map
$\phi\mapsto\bd{\phi}$ is linear and increasing, and hence every
$F$ of the form \eqref{eq:riesz-cvx-main} is linear and increasing.
Condition \ref{enu:cond-finitness} of Theorem \ref{thm:riesz-cvx-g}
is also simple to check: We take $\phi_{0}(x)=\left|x\right|$ and
observe that $\phi_{0}^{\ast}=\oo_{B_{2}^{n}}^{\infty}$ so $\int e^{-\phi_{0}^{\ast}}=\left|B_{2}^{n}\right|>0$.
Moreover the assumptions on $\mu$ and $\nu$ guarantee that 
\[
F(\phi_{0})=\int_{\RR^{n}}\left|x\right|\dd\mu+\int_{\SS^{n-1}}\dd\nu
\]
 is finite.

All that remains is to check condition \ref{enu:cond-cont} of Theorem
\ref{thm:riesz-cvx-g}. We do so in the following proposition:
\begin{prop}
\label{prop:continuity}If $\left\{ \phi_{i}\right\} _{i=1}^{\infty},\phi\in\cvx_{n}$
and $\phi_{i}^{\ast}\downarrow\phi^{\ast}$ then $\phi_{i}\uparrow\phi$
and $\bd{\phi_{i}}\uparrow\bd{\phi}$. Hence $F(\phi_{i})\to F(\phi)$
for every functional $F$ of the form \eqref{eq:riesz-cvx-main}.
\end{prop}

\begin{proof}
Since $\phi_{i}^{\ast}\downarrow\phi^{\ast}$ obviously $\phi_{i}^{\ast}\ge\phi^{\ast}$
for all $i$, so $\phi_{i}\le\phi$ for all $i$.

In the other direction, fix $x\in\RR^{n}$ with $\phi(x)<\infty$
and fix $\epsilon>0$. Choose $y_{\epsilon}\in\RR^{n}$ such that
\[
\phi(x)=\sup_{y\in\RR^{n}}\left(\left\langle x,y\right\rangle -\phi^{\ast}(y)\right)\le\left\langle x,y_{\epsilon}\right\rangle -\phi^{\ast}(y_{\epsilon})+\epsilon.
\]
 Then for every $i\ge1$ we have 
\begin{align*}
\phi_{i}(x) & \ge\left\langle x,y_{\epsilon}\right\rangle -\phi_{i}^{\ast}(y_{\epsilon})\ge\left(\phi(x)+\phi^{\ast}(y_{\epsilon})-\epsilon\right)-\phi_{i}^{\ast}(y_{\epsilon})\\
 & =\phi(x)+\left(\phi^{\ast}(y_{\epsilon})-\phi_{i}^{\ast}(y_{\epsilon})\right)-\epsilon.
\end{align*}
 Letting $i\to\infty$ and using the fact that $\phi_{i}^{\ast}(y_{\epsilon})\to\phi^{\ast}(y_{\epsilon})$
we see that $\lim_{i\to\infty}\phi_{i}(x)\ge\phi(x)-\epsilon$. As
$\epsilon>0$ was arbitrary we conclude that $\phi_{i}(x)\uparrow\phi(x)$.
The case $\phi(x)=\infty$ is handled similarly. 

We now prove that $\bd{\phi_{i}}\uparrow\bd{\phi}$. We fix a point
$p\in\RR^{n}$ such that $\phi(p)<\infty$ , and therefore $\phi_{i}(p)<\phi(p)<\infty$
for all $i$. Then we can write
\[
\bd{\phi_{i}}(\theta)=\lim_{\lambda\to\infty}\frac{\phi_{i}(p+\lambda\theta)}{\lambda}=\lim_{\lambda\to\infty}\frac{\phi_{i}(p+\lambda\theta)-\phi_{i}(p)}{\lambda}=\sup_{\lambda>0}\frac{\phi_{i}(p+\lambda\theta)-\phi_{i}(p)}{\lambda},
\]
 where the last equality holds by the convexity of $\phi_{i}$. The
same of course holds for $\phi$ instead of $\phi_{i}$. Therefore
we have 
\begin{align*}
\lim_{i\to\infty}\bd{\phi_{i}}(\theta) & =\sup_{i\in\NN}\bd{\phi_{i}}(\theta)=\sup_{i\in\NN}\sup_{\lambda>0}\frac{\phi_{i}(p+\lambda\theta)-\phi_{i}(p)}{\lambda}=\sup_{\lambda>0}\sup_{i\in\NN}\frac{\phi_{i}(p+\lambda\theta)-\phi_{i}(p)}{\lambda}\\
 & \ge\sup_{\lambda>0}\left(\lim_{i\to\infty}\frac{\phi_{i}(p+\lambda\theta)-\phi_{i}(p)}{\lambda}\right)=\sup_{\lambda>0}\frac{\phi(p+\lambda\theta)-\phi(p)}{\lambda}=\bd{\phi}(\theta).
\end{align*}
 Since we clearly have $\bd{\phi_{i}}\le\bd{\phi}$ for all $i$,
we conclude that indeed $\bd{\phi_{i}}\uparrow\bd{\phi}$.

Finally, the monotone convergence theorem implies that every functional
$F$ of the form \eqref{eq:riesz-cvx-main} must satisfy $F(\phi_{i})\to F(\phi)$. 
\end{proof}
For most of the proof of Theorem \ref{thm:riesz-cvx-g} we will not
work with the full class $\cvx_{n}$. Instead, we will work with the
following class:
\begin{defn}
A function $\phi\in\cvx_{n}$ is of linear growth if there exists
constants $A_{\phi},B_{\phi}>0$ such that $\phi(x)\le A_{\phi}\left|x\right|+B_{\phi}$
for all $x\in\RR^{n}$. We denote the class of all convex functions
of linear growth by $\cvx_{n}^{\ell}$. 
\end{defn}

If $\phi\in\cvx_{n}^{\ell}$ then clearly $\bd{\phi}$ is everywhere
finite. In fact a little more is true:
\begin{prop}
\label{prop:infty-uniform}Assume $\phi\in\cvx_{n}^{\ell}$. Then
the limit $\widehat{\phi}(\theta)=\lim_{\lambda\to\infty}\frac{\phi(\lambda\theta)}{\lambda}$
exists uniformly in $\theta\in\SS^{n-1}$.
\end{prop}

\begin{proof}
The main point here is that a convex function of linear growth is
(globally) Lipschitz. To see this, fix $x,y\in\RR^{n}$. For $t>0$
we define $z_{t}=y+t(y-x)$ so that $y\in[x,z_{t}]$. By convexity
we then have 
\[
\frac{\phi(y)-\phi(x)}{\left|y-x\right|}\le\frac{\phi(z_{t})-\phi(x)}{\left|z_{t}-x\right|}\le\frac{\left(A_{\phi}\left|z_{t}\right|+B_{\phi}\right)-\phi(x)}{\left|z_{t}\right|-\left|x\right|}\xrightarrow{t\to\infty}A_{\phi},
\]
 so $\phi$ is $A_{\phi}$-Lipschitz. It follows that the family of
functions $\phi_{\lambda}:\SS^{n-1}\to\RR$ defined by $\phi_{\lambda}(\theta)=\frac{\phi(\lambda\theta)}{\lambda}$
are all Lipschitz on $\SS^{n-1}$ with a uniform constant $A_{\phi}$.
In particular the family $\left\{ \phi_{\lambda}\right\} _{\lambda>0}$
is a equicontinuous, which implies that the convergence $\phi_{\lambda}\to\bd{\phi}$
is uniform. 
\end{proof}
Finally, we will need the following approximation lemma:
\begin{lem}
\label{lem:lin-growth-approx}For every $\phi\in\cvx_{n}$ there exist
a sequence $\left\{ \phi_{k}\right\} _{k=1}^{\infty}\subseteq\cvx_{n}^{\ell}$
such that $\phi_{k}^{\ast}\downarrow\phi^{\ast}$. 
\end{lem}

\begin{proof}
We define 
\[
\phi_{k}(x)=\inf_{y\in\RR^{n}}\left[\phi(y)+k\left|x-y\right|\right].
\]
 Since $\phi_{k}(x)\le\phi(0)+k\left|x\right|$ it follows that every
$\phi_{k}$ is of linear growth. A computation shows that $\phi_{k}^{\ast}=\phi^{\ast}+\oo_{kB_{2}^{n}}^{\infty}$,
so we indeed have $\phi_{k}^{\ast}\downarrow\phi^{\ast}$.
\end{proof}

\section{\label{sec:non-finite}Non-Finite linear functionals}

We now prove Theorem \ref{thm:riesz-cvx-g}. We begin with the following
lemma:
\begin{lem}
\label{lem:lin-finite}Assume $F:\cvx_{n}\to(-\infty,\infty]$ satisfies
the assumptions of Theorem \ref{thm:riesz-cvx-g}. Then $F(\phi)<\infty$
for every $\phi$ of linear growth.
\end{lem}

\begin{proof}
We are given that there exists a function $\phi_{0}\in\cvx_{n}$ such
that $\int e^{-\phi_{0}^{\ast}}>0$ and $F(\phi_{0})<\infty$. If
$\mathbf{0}$ denotes the constant function $0$ then 
\[
F(\phi_{0})=F\left(\mathbf{0}+\phi_{0}\right)=F\left(\mathbf{0}\right)+F\left(\phi_{0}\right),
\]
 so $F\left(\mathbf{0}\right)=0$. For every affine function $\ell$
we have $\pm\ell\in\cvx_{n}$ and $F(\ell)+F(-\ell)=F\left(\mathbf{0}\right)=0$,
so $F\left(\ell\right)<\infty$. 

Since $\int e^{-\phi_{0}^{\ast}}>0$ we can find a ball $B(p,r)$
such that $\phi_{0}^{\ast}$ is finite and bounded on $B(p,r$), say
$\phi_{0}^{\ast}<m$. Then 
\begin{align*}
\phi_{0}(x) & =\sup_{y\in\RR^{n}}\left[\left\langle x,y\right\rangle -\phi_{0}^{\ast}(y)\right]\ge\left\langle x,p+r\frac{x}{2\left|x\right|}\right\rangle -\phi_{0}^{\ast}\left(p+r\frac{x}{2\left|x\right|}\right)\\
 & \ge\left\langle x,p\right\rangle +\frac{r}{2}\left|x\right|-m.
\end{align*}
 It follows that 
\begin{align*}
F\left(\left|x\right|\right) & =\frac{2}{r}\left(F\left(\frac{r}{2}\left|x\right|+\left\langle x,p\right\rangle -m\right)+F\left(-\left\langle x,p\right\rangle +m\right)\right)\\
 & \le\frac{2}{r}\left(F(\phi_{0})+F\left(-\left\langle x,p\right\rangle +m\right)\right)<\infty.
\end{align*}
 Finally, for every $\phi\in\cvx_{n}^{\ell}$ we have 
\[
F\left(\phi\right)\le F\left(A_{\phi}\left|x\right|+B_{\phi}\right)=A_{\phi}F\left(\left|x\right|\right)+F\left(B_{\phi}\right)<\infty
\]
 as we claimed.
\end{proof}
We now proceed in a way similar to the proof of Theorem \ref{thm:riesz-cvxf}:
We consider the space 
\[
E_{\ell}=\left\{ \phi+f:\ \phi\in\cvx_{n}^{\ell}\text{ and }f\in C_{c}\left(\RR^{n}\right)\right\} ,
\]
and claim the following extension result:
\begin{lem}
\label{lem:El-extend}Let $F:\cvx_{n}^{\ell}\to\RR$ be a linear and
increasing functional. Then $F$ can be extended to a functional $F:E_{\ell}\to\RR$
which is again linear and increasing.
\end{lem}

The proof of Lemma \ref{lem:El-extend} is identical to the proof
of Lemma \ref{lem:E-extend}: just replace $\cvx_{n}^{F}$ by $\cvx_{n}^{\ell}$
everywhere in the proof. It should now be clear why we chose in the
proof to work the functions $\rho_{a}$ (which satisfy $\rho_{a}\in\cvx_{n}^{\ell}$)
and not with simpler functions like $a\frac{\left|x\right|^{2}}{2}$
(which do not). 

We are ready to prove a Riesz type theorem for the space $\cvx_{n}^{\ell}$:
\begin{thm}
\label{thm:repr-lin-growth}Let $F:\cvx_{n}^{\ell}\to\RR$ be a linear
and increasing functional. Then there exists a unique finite Borel
measure $\mu$ on $\RR^{n}$ with a finite first moment, and a unique
finite Borel measure $\nu$ on $\SS^{n-1}$, such that 
\[
F(\phi)=\int_{\RR^{n}}\phi\dd\mu+\int_{\SS^{n-1}}\bd{\phi}\dd\nu
\]
 for all $\phi\in\cvx_{n}$.
\end{thm}

\begin{proof}
First, by Lemma \ref{lem:El-extend} we can extend $F$ to the space
$E_{\ell}$. Applying the classical Riesz theorem to the restriction
$\left.F\right|_{C_{c}\left(\RR^{n}\right)}$, we conclude that there
exists a Random measure $\mu$ on $\RR^{n}$ such that $F(f)=\int_{\RR^{n}}f\dd\mu$
for all $f\in C_{c}\left(\RR^{n}\right)$. The same argument as in
Theorem \ref{thm:riesz-cvxf} shows that $F(\phi)\ge\int_{\RR^{n}}\phi\dd\mu$
for all non-negative $\phi\in\cvx_{n}^{\ell}$ (we will essentially
repeat the argument in the next paragraph). In particular 
\[
\int\dd\mu\le F\left(\oo\right)<\infty,\qquad\int\left|x\right|\dd\mu\le F\left(\left|x\right|\right)<\infty,
\]
 which implies that $\mu$ is finite with a finite first moment. 

Define $G:\cvx_{n}^{\ell}\to\RR$ by 
\[
G\left(\phi\right)=F(\phi)-\int_{\RR^{n}}\phi\dd\mu.
\]
 Note that $G$ is indeed always finite, since $\mu$ is finite with
finite first moment and $\phi$ is of linear growth. We will need
two properties of the functional $G$. First, we we claim that $G$
is itself linear and increasing. Indeed, $G$ is clearly linear. To
show that $G$ is increasing we fix two functions $\phi_{1},\phi_{2}\in\cvx_{n}^{\ell}$
such that $\phi_{1}\le\phi_{2}$. For every $R>0$ let $\eta_{R}\in C_{c}\left(\RR^{n}\right)$
be a function such that $0\le\eta_{R}\le1$, $\eta_{R}\equiv1$ on
the ball $B(0,R)$, and $\eta_{R}$ is compactly supported. Since
$f\cdot(\phi_{2}-\phi_{1})\in C_{c}\left(\RR^{n}\right)$ and $f\cdot(\phi_{2}-\phi_{1})+\phi_{1}\le\phi_{2}$
we conclude that 
\begin{align*}
F(\phi_{2}) & \ge F\left(f(\phi_{2}-\phi_{1})\right)+F(\phi_{1})=\int_{\RR^{n}}f\left(\phi_{2}-\phi_{1}\right)\dd\mu+F(\phi_{1})\\
 & \ge\int_{B(0,R)}\left(\phi_{2}-\phi_{1}\right)\dd\mu+F(\phi_{1}).
\end{align*}
 Letting $R\to\infty$ and using the monotone convergence theorem
we conclude that 
\[
F(\phi_{2})\ge\int_{\RR^{n}}\left(\phi_{2}-\phi_{1}\right)\dd\mu+F(\phi_{1}),
\]
or $G(\phi_{2})\ge G(\phi_{1})$. 

Next, we want to prove that $G(\phi)$ depends only on $\bd{\phi}$.
In other words, we want to prove that if $\bd{\phi_{1}}=\bd{\phi_{2}}$
then $G(\phi_{1})=G(\phi_{2})$. We first show it under the extra
assumption that $\phi_{1}\le\phi_{2}$. 

Under this assumption we already proved that $G(\phi_{1})\le G(\phi_{2})$
. For the reverse inequality, we again fix $R>0$ and define 
\[
\zeta_{R}(x)=\max\left(0,R\left(\phi_{2}(x)-\phi_{1}(x)\right)-\left|x\right|\right).
\]
 We claim that $\zeta_{R}\in C_{c}\left(\RR^{n}\right)$. Since it
is obviously continuous, we just need to show it is compactly supported.
By our assumption we have 
\[
\lim_{\lambda\to\infty}\frac{\phi_{1}(\lambda\theta)}{\lambda}=\bd{\phi_{1}}(\theta)=\bd{\phi_{2}}(\theta)=\lim_{\lambda\to\infty}\frac{\phi_{2}(\lambda\theta)}{\lambda},
\]
 and by Proposition \ref{prop:infty-uniform} these limits are uniform
in $\theta\in\SS^{n-1}$. It follows that 
\[
\lim_{\lambda\to\infty}\frac{\phi_{2}(\lambda\theta)-\phi_{1}(\lambda\theta)}{\lambda}=0
\]
 uniformly in $\theta\in\SS^{n-1}$, or equivalently $\phi_{2}(x)-\phi_{1}(x)=o\left(\left|x\right|\right)$
as $\left|x\right|\to\infty$. This shows that $\zeta_{R}$ is compactly
supported. 

We can bound $\zeta_{R}$ from above and from below. From above, it
is clear that $\zeta_{R}\le R\left(\phi_{2}-\phi_{1}\right)$. From
below, we have $\zeta_{R}+\left|x\right|+R\phi_{1}\ge R\phi_{2}.$
It follows that 
\begin{align*}
F\left(\phi_{2}\right) & =\frac{1}{R}F\left(R\phi_{2}\right)\le\frac{1}{R}F\left(\zeta_{R}+\left|x\right|+R\phi_{1}\right)\\
 & =\frac{1}{R}\left(F\left(\zeta_{R}\right)+F\left(\left|x\right|\right)+R\cdot F(\phi_{1})\right)\\
 & =\frac{1}{R}\left(\int\zeta_{R}\dd\mu+F\left(\left|x\right|\right)+R\cdot F(\phi_{1})\right)\\
 & \le\int\left(\phi_{2}-\phi_{1}\right)\dd\mu+\frac{1}{R}F\left(\left|x\right|\right)+F(\phi_{1}).
\end{align*}
 Letting $R\to\infty$ we see that indeed $G(\phi_{2})\le G(\phi_{1})$
. 

This concludes the proof that $G(\phi_{1})=G(\phi_{2})$ under the
assumptions $\bd{\phi_{1}}=\bd{\phi_{2}}$ and $\phi_{1}\le\phi_{2}$.
But we can now get rid of this second assumption: If $\bd{\phi_{1}}=\bd{\phi_{2}}$
then we also have 
\[
\bd{\max\left(\phi_{1},\phi_{2}\right)}=\max\left(\bd{\phi_{1}},\bd{\phi_{2}}\right)=\bd{\phi_{1}}=\bd{\phi_{2}},
\]
 and since $\phi_{1},\phi_{2}\le\max\left(\phi_{1},\phi_{2}\right)$
it follows that $G(\phi_{1})=G\left(\max\left(\phi_{1},\phi_{2}\right)\right)=G(\phi_{2})$
. Hence $G(\phi)$ depends only on $\overline{\phi}$. 

Consider now the spaces 
\begin{align*}
E_{1} & =\left\{ h_{K}:\ K\in\KK^{n}\right\} \\
E_{2} & =\left\{ \bd{\phi}:\ \phi\in\cvx_{n}^{\ell}\right\} .
\end{align*}
 We have $E_{1}\subseteq E_{2}\subseteq C\left(\SS^{n-1}\right)$
since every support function $h_{K}:\SS^{n-1}\to\RR$ can also be
though of as a $1$-homogeneous convex function on $\RR^{n}$, and
under this identification we have $\bd{h_{K}}=h_{K}$. Note that we
can define $\widetilde{G}:E_{2}\to\RR$ by $\widetilde{G}(\bd{\phi})=G(\phi)$,
since we just proved that this definition does not depend on the choice
of $\phi$. 

Since the restriction $\left.\widetilde{G}\right|_{E_{1}}:E_{1}\to\RR$
is clearly linear and increasing, Theorem \ref{thm:riesz-bodies}
implies that exists a finite Borel measure $\nu$ on $\SS^{n-1}$
such that 
\begin{equation}
\widetilde{G}(h_{k})=\int_{\SS^{n-1}}h_{K}\dd\nu.\label{eq:body-repr}
\end{equation}
Moreover, in the proof of Theorem \ref{thm:riesz-bodies} we saw that
$\left.\widetilde{G}\right|_{E_{1}}$ has a unique extension to $C\left(\SS^{n-1}\right)$
which is linear and increasing, and by inspecting the proof we see
that an extension to $E_{2}$ also has to be unique. But $\widetilde{G}$
itself is such an extension, and so is the map $\bd{\phi}\mapsto\int_{\SS^{n-1}}\bd{\phi}\dd\nu$.
Therefore they must coincide, which means that 
\[
G(\phi)=\widetilde{G}(\overline{\phi})=\int_{\SS^{n-1}}\bd{\phi}\dd\nu
\]
 for all $\phi\in\cvx_{n}^{\ell}$. But then 
\[
F(\phi)=\int_{\RR^{n}}\phi\dd\mu+G(\phi)=\int_{\RR^{n}}\phi\dd\mu+\int_{\SS^{n-1}}\bd{\phi}\dd\nu,
\]
 finishing the proof of existence. 

Finally, we need to show that $\mu$ and $\nu$ are uniquely defined.
Towards this goal assume that 
\[
\int_{\RR^{n}}\phi\dd\mu_{1}+\int_{\SS^{n-1}}\bd{\phi}\dd\nu_{1}=\int_{\RR^{n}}\phi\dd\mu_{2}+\int_{\SS^{n-1}}\bd{\phi}\dd\nu_{2}
\]
 for all $\phi\in\cvx_{n}^{\ell}$. Fix $K\in\KK^{n}$ with $0\in K$
and for $R>0$ define $\phi_{R}=\max\left\{ h_{K}-R,0\right\} \in\cvx_{n}^{\ell}$.
Then $\bd{\phi_{R}}=h_{K}$, so 
\[
\int_{\RR^{n}}\phi_{R}\dd\mu_{1}+\int_{\SS^{n-1}}h_{K}\dd\nu_{1}=\int_{\RR^{n}}\phi_{R}\dd\mu_{2}+\int_{\SS^{n-1}}h_{K}\dd\nu_{2}
\]
 for all $R>0$. But clearly $\phi_{R}\to0$ pointwise as $R\to\infty$
so by the monotone convergence theorem it follows that $\int_{\SS^{n-1}}h_{K}\dd\nu_{1}=\int_{\SS^{n-1}}h_{K}\dd\nu_{2}$.
The same will also be true without the assumption that $0\in K$,
since for every $K\in\KK^{n}$ and every $r>0$ we have $h_{K}=h_{K+rB_{2}^{n}}-h_{rB_{2}^{n}}$,
and for large enough $r>0$ we do have $0\in K+rB_{2}^{n}$. It now
follows from the uniqueness part of Theorem \ref{thm:riesz-bodies}
that $\nu_{1}=\nu_{2}$. We can then repeat the argument from Theorem
\ref{thm:riesz-cvxf} to show that also $\mu_{1}=\mu_{2}$. 
\end{proof}
From here Theorem \ref{thm:riesz-cvx-g} (and Theorem \ref{thm:riesz-functions-general})
is a simple corollary:
\begin{proof}[Proof of Theorem \ref{thm:riesz-cvx-g}]
Given $F:\cvx_{n}\to(-\infty,\infty]$ , the restriction $\left.F\right|_{\cvx_{n}^{\ell}}$
satisfies the assumptions of Theorem \ref{thm:repr-lin-growth} (we
need Lemma \ref{lem:lin-finite} here to know that $\left.F\right|_{\cvx_{n}^{\ell}}$
is always finite).

From Theorem \ref{thm:repr-lin-growth} there exists $\mu$ and $\nu$
of the required form such that 
\[
F(\phi)=\int_{\RR^{n}}\phi\dd\mu+\int_{\SS^{n-1}}\bd{\phi}\dd\nu
\]
 for all $\phi\in\cvx_{n}^{\ell}$. For an arbitrary $\phi\in\cvx_{n}$
we use Lemma \ref{lem:lin-growth-approx} and choose a sequence $\left\{ \phi_{k}\right\} _{k=1}^{\infty}\subseteq\cvx_{n}^{\ell}$
such that $\phi_{k}^{\ast}\downarrow\phi^{\ast}$ . We then have 
\begin{align*}
F(\phi) & =\lim_{k\to\infty}F(\phi_{k})=\lim_{k\to\infty}\left(\int_{\RR^{n}}\phi_{k}\dd\mu+\int_{\SS^{n-1}}\bd{\phi_{k}}\dd\nu\right)\\
 & =\int_{\RR^{n}}\phi\dd\mu+\int_{\SS^{n-1}}\bd{\phi}\dd\nu,
\end{align*}
 where we used Proposition \ref{prop:continuity} in the last equality. 

Uniqueness of $\mu$ and $\nu$ is obvious since by Theorem \ref{thm:repr-lin-growth}
even the restriction $\left.F\right|_{\cvx_{n}^{\ell}}$ determines
$\mu$ and $\nu$ uniquely. 
\end{proof}
We conclude this section with two examples that show the extra conditions
imposed in Theorem \ref{thm:riesz-functions-general} are indeed necessary:
\begin{example}
Define $F:\lc_{n}\to(-\infty,\infty]$ by 
\[
F(f)=\begin{cases}
0 & \text{If }f=c\cdot\oo_{\left\{ p\right\} }\text{ for }c\ge0\text{ and }p\in\RR^{n}\\
+\infty & \text{Otherwise}.
\end{cases}
\]
$F$ is linear and increasing, but is not of the form \eqref{eq:riesz-func-repr}.
The reason of course is that $F$ is essentially ``lower dimensional'',
i.e. $F(f)=\infty$ for every $f\in\lc_{n}$ with $\int f>0$.
\end{example}

\begin{example}
Without the weak continuity assumption of Theorem \ref{thm:riesz-functions-general}
one can create less explicit counterexamples. For example, fix a free
ultrafilter $\mathcal{U}$ on the natural numbers $\NN$, a direction
$\theta\in\SS^{n-1}$ and a number $\alpha>1$ and consider the functional
$F:\lc_{n}\to(-\infty,\infty]$ defined by 
\[
F(f)=\lim_{\mathcal{U}}\frac{h_{f}(n\theta)}{n^{\alpha}}.
\]
 Since ultralimits are linear and increasing, it follows that $F$
is also linear and increasing. However, $F$ is not of the form \eqref{eq:riesz-func-repr}
as it is not continuous: For every $f\in\lc_{n}$ with $h_{f}\in\cvx_{n}^{\ell}$
we have $F(f)=0$, but $F$ is not identically $0$. 
\end{example}

\section{\label{sec:surface-area}Surface area measures}

As was explained in the introduction, linear functionals appear naturally
in convex geometry as the derivatives of more general functionals.
For convex bodies, the simplest choice of such a functional was the
volume, i.e. $G(K)=\left|K\right|$. For log-concave functions, the
simplest choice is to consider the functional $G:\lc_{n}\to[0,\infty]$
defined by $G(f)=\int f$. We now take its derivative:
\begin{defn}
For $f,g\in\lc_{n}$ such that $0<\int g<\infty$, we define 
\[
\delta(g,f)=\lim_{t\to0^{+}}\frac{\int\left(g\star\left(t\cdot f\right)\right)-\int g}{t}.
\]
\end{defn}

This first variation of the integral was first systematically studied
by Colesanti and Fragalà in \cite{Colesanti2013}, who showed that
the limit in the definition exists and belongs to $(-\infty,\infty]$
. 

If we now consider the functional $F_{g}:\lc_{n}\to(-\infty,\infty]$
defined by $F_{g}(f)=\delta(g,f)$, it is clear from the definition
that $F_{g}$ is increasing. In the case of convex bodies we have
the theory of mixed volumes available to us, which implies that functionals
analogous to $F_{g}$ are also linear. For log-concave functions there
is no such theory, so while we expect $F_{g}$ to be linear as a directional
derivative, this is not obvious. For this reason one cannot use a
result such as Theorem \ref{thm:riesz-functions-general} directly.
However, Theorem \ref{thm:riesz-functions-general} still gives us
a clue for the type of formulas we expect. Under technical conditions
such a formula was proved in \cite{Colesanti2013}. We now describe
this result.
\begin{defn}
\label{def:sa-measure}Fix $g\in\lc_{n}$ with $0<\int g<\infty$
and write $\psi=-\log g$. We define the Borel measure $\mu_{g}$
on $\RR^{n}$ by $\left(\nabla\psi\right)_{\sharp}\left(g\dd x\right)$,
where $\sharp$ denotes the push-forward. Similarly the measure $\nu_{g}$
is the Borel measure on $\SS^{n-1}$ defined by 
\[
\nu_{g}=\left(n_{K_{g}}\right)_{\sharp}\left(\left.g\HH^{n-1}\right|_{\partial K_{g}}\right).
\]
Here $n_{K_{g}}:\partial K_{g}\to\SS^{n-1}$ is the Gauss map (mapping
every point $x\in\partial K_{g}$ to the normal to $K_{g}$ at the
point $x$), and $\HH^{n-1}$ denotes the $(n-1)$-Hausdorff measure. 
\end{defn}

Note that we need no regularity assumptions on $g\in\lc_{n}$ in order
to define $\mu_{g}$ and $\nu_{g}$. Indeed, it is well-known $\nabla\psi$
exists (Lebesgue) almost everywhere on $K_{g}$, so the push-forward
is always well defined. Similarly, the Gauss map $n_{K_{g}}$ is always
defined $\mathcal{H}^{n-1}$-almost everywhere on $\partial K_{g}$.
However, the theorem of \cite{Colesanti2013} does require significant
regularity assumptions on our functions.

For the statement of the theorem, let us say that a function $f\in\lc_{n}$
is \emph{sufficiently regular }if:
\begin{enumerate}
\item The support $K_{f}$ of $f$ is a $C^{2}$ smooth convex body with
everywhere positive Gauss curvature.
\item The function $\psi=-\log f$ is continuous in $K_{f}$, $C^{2}$-smooth
in the interior of $K_{f}$, and has a strictly positive-definite
Hessian. 
\item $\lim_{x\to\partial K_{f}}\left|\nabla\psi(x)\right|=\infty$ .
\end{enumerate}
We can now state:
\begin{thm}[Colesanti–Fragalà]
\label{thm:colesanti-fragala}Assume $f,g\in\lc_{n}$ are sufficiently
regular. Assume further that $h_{g}-c\cdot h_{f}$ is convex for some
$c>0$. Then 
\[
\delta(g,f)=\int_{\RR^{n}}h_{f}\dd\mu_{g}+\int_{\SS^{n-1}}h_{K_{f}}\dd\nu_{g}.
\]
 
\end{thm}

From this theorem we see that functionals of the form \eqref{eq:riesz-func-repr}
do appear ``in nature''. This explains why are interested in results
like Theorem \ref{thm:riesz-functions-general} and are not satisfied
with the simpler Theorem \ref{thm:riesz-functions-sc}.

The technical assumptions in Theorem \ref{thm:colesanti-fragala}
are known to be non-optimal. Proving the same result for every $f,g\in\lc_{n}$
is an interesting problem outside the scope of this paper. However,
we do understand completely the case $\nu_{g}\equiv0$:
\begin{thm}[\cite{Rotem2021}]
\label{thm:repr-ess}Fix $g\in\lc_{n}$ such that $0<\int g<\infty$.
Then the following are equivalent:
\begin{enumerate}
\item $g$ is essentially continuous, i.e. the set $\left\{ x\in\RR^{n}:\ g\text{ is not continuous at x}\right\} $
has zero $\mathcal{H}^{n-1}$ measure. 
\item For every $f\in\lc_{n}$ one has $\delta(g,f)=\int_{\RR^{n}}h_{f}\dd\mu_{g}$.
\end{enumerate}
\end{thm}

To better understand the notion of essential continuity, note that
as a log-concave function $g$ is automatically continuous outside
of $\partial K_{g}$. Moreover, since $g$ is upper semi-continuous
it is easy to check that $g$ is continuous at a point $x\in\partial K_{g}$
if and only if $g(x)=0$. Therefore $g$ is essentially continuous
if and only if $g$ vanishes $\mathcal{H}^{n-1}$-almost everywhere
on $\partial K_{g}$. Equivalently, $g$ is essentially continuous
if and only if $\nu_{g}\equiv0$. 

The importance of essential continuity in the context of surface area
measures was first realized by Cordero-Erausquin and Klartag in \cite{Cordero-Erausquin2015}.
In this paper the authors studied the measure $\mu_{g}$ from Definition
\ref{def:sa-measure} in a different language: If $g=e^{-\psi}$ then
$\mu_{g}$ is called there the moment measure of $\psi$. The main
result of \cite{Cordero-Erausquin2015} is the following theorem which
characterizes measures of the form $\mu_{g}$:
\begin{thm}[Cordero-Erausquin, Klartag]
\label{thm:minkowski}Let $\mu$ be a Borel measure on $\RR^{n}$.
Then $\mu=\mu_{g}$ for an essentially continuous function $g\in\lc_{n}$
if and only if $\mu$ satisfies the following properties:
\begin{enumerate}
\item $0<\mu\left(\RR^{n}\right)<\infty$ .
\item $\mu$ has a finite first moment, and $\int_{\RR^{n}}x\dd\mu=0$. 
\item $\mu$ is not concentrated on any lower dimensional subspace (i.e.
for every proper linear subspace $H$ of $\RR^{n}$ we have $\mu(H)<\mu\left(\RR^{n}\right)$
). 
\end{enumerate}
\end{thm}

Santambrogio gave in \cite{Santambrogio2016} an alternative proof
of Theorem \ref{thm:minkowski} using methods of optimal transportation. 

Our Theorem \ref{thm:riesz-functions-sc} can be combined with the
known results mentioned above in order to characterize linear increasing
functionals as the first variation of the integral. Here is an example
of such a result:
\begin{thm}
Let $F:\lc_{n}^{sc}\to\RR$ be a linear and increasing functional.
Assume further that:
\begin{enumerate}
\item $F\left(\oo_{\left\{ p\right\} }\right)=0$ for every $p\in\RR^{n}$. 
\item For every line $\ell$ through the origin $F\left(e^{-\frac{1}{2}\left|x\right|^{2}}\oo_{\ell}\right)\ne0$. 
\end{enumerate}
Then there exists an essentially continuous log-concave function $g\in\lc_{n}$
such that $F(f)=\delta(g,f)$ for all $f\in\lc_{n}^{sc}$. 
\end{thm}

Of course, the choice of the function $e^{-\frac{1}{2}\left|x\right|^{2}}\oo_{\ell}$
is rather arbitrary. As will be clear from the proof, this is a non-degeneracy
condition meant to exclude ``lower dimensional'' examples. 
\begin{proof}
Since $F$ is linear and increasing, by Theorem \ref{thm:riesz-functions-sc}
there exists a finite compactly supported Borel measure $\mu$ such
that $F(f)=\int_{\RR^{n}}h_{f}\dd\mu$ for all $f\in\lc_{n}^{sc}$. 

For every $p\in\RR^{n}$ we have 
\[
\left\langle \int_{\RR^{n}}x\dd\mu,p\right\rangle =\int_{\RR^{n}}\left\langle x,p\right\rangle \dd\mu=\int_{\RR^{n}}h_{\oo_{\left\{ p\right\} }}\dd\mu=F\left(\oo_{\left\{ p\right\} }\right)=0
\]
 by our assumptions. Hence $\int x\dd\mu=0$. 

Next, fix a hyperplane $H\subseteq\RR^{n}$ and define $\ell=H^{\perp}$.
Note that if $f=e^{-\frac{1}{2}\left|x\right|^{2}}\oo_{\ell}$ then
\[
h_{f}(x)=\sup_{y\in\ell}\left[\left\langle x,y\right\rangle -\frac{1}{2}\left|y\right|^{2}\right]=\frac{1}{2}\left|\text{Pr}_{\ell}x\right|^{2},
\]
 where $\text{Pr}_{\ell}$ denotes the orthogonal projection onto
$\ell$. In particular $h_{f}\equiv0$ on $H$. Therefore 
\[
\int_{\RR^{n}}h_{f}\dd\mu=F(f)\ne0=\int_{H}h_{f}\dd\mu,
\]
 so $\mu$ is not supported on $H$. 

We see that $\mu$ satisfies all the assumptions of Theorem \ref{thm:minkowski},
so there exists an essentially continuous $g\in\lc_{n}$ such that
$\mu=\mu_{g}$. But then by Theorem \ref{thm:repr-ess} we have for
every $f\in\lc_{n}$ 
\[
F(f)=\int_{\RR^{n}}h_{f}\dd\mu=\int_{\RR^{n}}h_{f}\dd\mu_{g}=\delta(g,f),
\]
 proving the result.
\end{proof}
A natural question is whether this result can be extended to functionals
$F:\lc_{n}\to(-\infty,\infty]$. However, in order to prove such a
result one first needs to solve the following problem:
\begin{problem}
Fix a finite Borel measure $\mu$ on $\RR^{n}$ and a finite Borel
measure $\nu$ on $\SS^{n-1}$. Under what conditions on $\mu$ and
$\nu$ can one find a function $g\in\lc_{n}$ with $\mu_{g}=\mu$
and $\nu_{g}=\nu$? 
\end{problem}

This is a very natural question, but it is much beyond the scope of
this paper and is better left for future research.

\bibliographystyle{plain}
\bibliography{../../citations/library}
 
\end{document}